\documentclass[letterpaper,10pt,conference]{ieeeconf}

\IEEEoverridecommandlockouts
\usepackage{amsmath}
\usepackage{amssymb}
\usepackage{tikz}
\usetikzlibrary{patterns,arrows, math}
\usepackage{pgfplots}

\usepackage{cite}
\usepackage{textcomp}   
\def\BibTeX{{\rm B\kern-.05em{\sc i\kern-.025em b}\kern-.08em
    T\kern-.1667em\lower.7ex\hbox{E}\kern-.125emX}}

\usepackage{subcaption}
\usepackage{comment}
\usepackage{enumerate}

\usepackage{tikz}
\usetikzlibrary{shapes, arrows}
\tikzstyle{block} = [draw, fill=white, rectangle, minimum height=0em, minimum width=0em]
\tikzstyle{output} = [coordinate]
\tikzstyle{input} = [coordinate]
\graphicspath{{resources/}}

\newtheorem{thm}{Theorem}[section]
\newtheorem{ass}{Assumption}
\newtheorem{defn}[thm]{Definition}

\newtheorem{lem}[thm]{Lemma}
\newtheorem{prop}[thm]{Proposition}

\newtheorem{algo}[thm]{Algorithm}

\newcommand{\setdef}[2]{\left\{\, #1\, \left|\, \vphantom{#1} #2 \right.\right\}}

\newcommand{\ddt}{\tfrac{\text{\normalfont d}}{\text{\normalfont d}t}}

\newcommand{\R}{\mathbb{R}}

\newcommand{\N}{\mathbb{N}}

\newcommand{\cD}{\mathcal{D}}
\newcommand{\cF}{\mathcal{F}}

\newcommand{\cN}{\mathcal{N}}

\newcommand{\cU}{\mathcal{U}}

\newcommand{\ve}{\varepsilon}
\newcommand{\eps}{\varepsilon}
\newcommand{\rp}{\mathbb{R}_{\geq 0}}

\newcommand{\fa}{\ \forall \, }
\newcommand{\ex}{\ \exists \, }
\newcommand{\rbl}{\left (}
\newcommand{\rbr}{\right )}
\newcommand{\sbl}{\left [}
\newcommand{\sbr}{\right ]}

\newcommand{\nl}{\left\|}
\newcommand{\nr}{\right\|}
\newcommand{\cbl}{\left\lbrace }
\newcommand{\cbr}{\right\rbrace }

\newcommand{\Norm}[2][ ]{\nl #2 \nr_{#1}}
\newcommand{\SNorm}[1]{\Norm[\infty]{#1}}

\newcommand{\GL}{\text{GL}}

\newcommand{\Rp}{\R_{\geq0}}

\DeclareMathOperator{\esssup}{\rm ess\,sup}

\DeclareMathOperator*{\loc}{loc}
\DeclareMathOperator*{\rf}{ref}

\title{\LARGE \bf
Funnel MPC with feasibility constraints for nonlinear systems with arbitrary relative degree*
}

\author{Thomas Berger$^{1}$ and Dario Dennst\"adt$^{2}$
\thanks{*T.~Berger acknowledges funding by the Deutsche Forschungsgemeinschaft (DFG, German Research Foundation) -- Project-ID 471539468. D.~Dennst\"adt gratefully thanks the Technische Universit\"at Ilmenau and the Free State of Thuringia for their financial support as part of the Th\"uringer Graduiertenf\"orderung.}
\thanks{$^{1}$Thomas Berger is with the Universit\"at Paderborn, Institut f\"ur Mathematik, Warburger Str.~100, 33098~Paderborn, Germany
        {\tt\small thomas.berger@math.upb.de}}%
\thanks{$^{2}$Dario Dennst\"adt is with the Technische Universit\"at Ilmenau, Institut für Mathematik, Weimarer Str.~25, 98693 Ilmenau, Germany
        {\tt\small dario.dennstaedt@tu-ilmenau.de}}%
}

\begin{document}

\maketitle
\thispagestyle{empty}
\pagestyle{empty}

\begin{abstract}
We study tracking control for nonlinear systems with known relative degree and stable internal dynamics by the recently introduced technique of Funnel MPC. The objective is to achieve the evolution of the tracking error within a prescribed performance funnel. We propose a novel stage cost for Funnel MPC, extending earlier designs to the case of arbitrary relative degree, and show that the control objective as well as initial and recursive feasibility are always achieved~-- without requiring any terminal conditions or a sufficiently long prediction horizon. We only impose an additional feasibility constraint in the optimal control problem.
\end{abstract}

\section{Introduction}

In the recent work~\cite{berger2019learningbased} a novel model predictive control (MPC) scheme, so called Funnel MPC (FMPC), was proposed, which is able to achieve tracking with a prescribed performance of the tracking error. MPC is an established control technique which relies on the successive
solution of optimal control problems (OCPs), see e.g.~\cite{grune2017nonlinear,rawlings2017model}. Since it is able to take control and state constraints directly into account, it is nowadays widely used and helpful in various applications, see e.g.~\cite{QinBadg03}.

FMPC resolves the issue of requiring suitable terminal conditions (costs and constraints) in
the OCP (cf.~\cite{rawlings2017model} and the references therein) or a sufficiently long prediction
horizon (cf.~\cite{boccia2014stability}) in order to achieve recursive feasibility. This is achieved by a ``funnel-like'' stage cost, which penalizes the tracking error and grows unbounded when it approaches the funnel boundary. However, in the
FMPC scheme proposed in~\cite{berger2019learningbased} output constraints were
incorporated in the~OCP. It was then shown in~\cite{BergDenn21} that for the case of relative degree
one systems these constraints are superfluous and the funnel-inspired stage costs automatically
ensure initial and recursive feasibility. A generalization of these results to systems with relative
degree two was outlined in~\cite{Denn22}, however requiring a sufficiently long prediction horizon. In the present paper we extend the results
from~\cite{BergDenn21} to systems with arbitrary relative degree by designing a suitable stage cost
function, which is inspired by a recent funnel control design from~\cite{Berg22}. We emphasize that this extension is not straightforward, since the proof of initial and recursive feasibility relies on results from adaptive control, where the obstacle of higher relative degree is an omnipresent issue~\cite{Mors96}.

The concept of funnel control was developed in the seminal work~\cite{IlchRyan02b} (see also the recent survey in~\cite{BergIlch21}) and proved advantageous in a variety of applications such as control of industrial servo-systems~\cite{Hack17}, underactuated multibody systems~\cite{BergDrue21,BergOtto19}, peak inspiratory pressure~\cite{PompWeye15}, adaptive cruise control~\cite{BergRaue20} and even the control of infinite-dimensional systems such as a boundary controlled heat equation~\cite{ReisSeli15b}, a moving water tank~\cite{BergPuch22} and defibrillation processes of the human heart~\cite{BergBrei21}. We like to stress that, in contrast to MPC, funnel control does not use a model of the system, the funnel control input is determined by the instantaneous values of the system state and cannot ``plan ahead''. This often results in unnecessarily high control values and a rapidly changing control signal. Numerical simulations from~\cite{berger2019learningbased,BergDenn21} show that FMPC exhibits a considerably better controller performance than funnel control.

We like to note that together with the novel stage cost function that we propose for FMPC the OCP contains an
additional feasibility constraint at the point of the succeeding state evaluation (a similar
condition was present in~\cite{berger2019learningbased}) to guarantee recursive feasibility.
However, we do not incorporate the output constraints over the whole horizon in the~OCP.

\subsection{Nomenclature}

In the following let $\N$ denote the natural numbers, $\N_0 = \N \cup\{0\}$, and $\R_{\ge 0} =[0,\infty)$. By $\|x\|$ we denote the Euclidean norm of $x\in\R^n$, and $\GL_n(\R)$ is the group of invertible $\R^{n\times n}$ matrices. For some interval $I\subseteq\R$, some $V\subseteq\R^m$ and $k\in\N$, $L^\infty(I, \R^{n})$ $\big(L^\infty_{\rm loc} (I, \R^{n})\big)$ is the Lebesgue space of measurable, (locally) essentially bounded functions $f\colon I\to\R^n$ with norm $\|f \|_{\infty} = \esssup_{t \in I} \|f(t)\|$, $W^{k,\infty}(I,  \R^{n})$ is the Sobolev space of all functions
$f:I\to\R^n$ with $k$-th order weak derivative $f^{(k)}$ and $f,f^{(1)},\ldots,f^{(k)}\in L^\infty(I, \R^{n})$, and  $C^k(V,  \R^{n})$ is the set of  $k$-times continuously differentiable functions  $f:  V  \to \R^{n}$, with $C(V,  \R^{n}) := C^0(V,  \R^{n})$.

\subsection{System class}\label{Sec:SysClass}
We consider nonlinear systems of the form
\begin{equation}\label{eq:Sys}
    \begin{aligned}
        \dot{x}(t)  & = f(x(t)) + g(x(t)) u(t),\quad x(t^0)=x^0,\\
        y(t)        & = h(x(t)),
    \end{aligned}
\end{equation}
with~$t^0\in\R_{\ge 0}$, $x^0\in\R^n$, and nonlinear functions~$f:\R^n\to \R^n$, $g:\R^n\to \R^{n\times
m}$ and $h : \R^n \to \R^m$. For an input $u \in L^\infty_{\rm loc}(\R_{\geq 0},\R^m)$ the system~\eqref{eq:Sys} has a solution in the sense of \textit{Carath\'{e}odory}, that is a function $x:[t^0,\omega)\to\R^n$,
$\omega>t^0$, with $x(t^0)=x^0$ which is absolutely continuous and
satisfies the ODE in~\eqref{eq:Sys} for almost all $t\in [t^0,\omega)$. A solution $x$ is said to be maximal, if it has no right extension that is also a solution. The \textit{response} associated with $u$ is any maximal solution of~\eqref{eq:Sys} and denoted by~$x(\cdot;t^0,x^0,u)$; it is unique
if the right-hand side of~\eqref{eq:Sys} is locally Lipschitz in~$x$.

We recall the notion of relative degree for system~\eqref{eq:Sys}, see e.g.~\cite[Sec. 5.1]{Isid95}.
Assuming that $f,g,h$ are sufficiently smooth, the Lie derivative of~$h$ along~$f$ is defined by $\rbl L_f h\rbr(x) = h'(x) f(x)$, and successively we define~$L_f^k h = L_f (L_{f}^{k-1} h)$ with~$L_f^0 h = h$.
Furthermore, for the matrix-valued function~$g$ we have
\[
    (L_gh)(x) = \sbl (L_{g_1}h)(x), \ldots, (L_{g_m}h)(x) \sbr,
\]
where~$g_i$ denotes the~$i$-th column of~$g$ for $i=1,\ldots, m$. Then system~\eqref{eq:Sys} is said to have \emph{(global) relative degree}~$r \in \mathbb{N}$, if
\begin{align*}
    \fa k \in \{1,\ldots,r-1\}\fa x \in \R^n:\
        (L_g L_f^{k-1} h)(x)  = 0&  \\
      \text{and}\quad
        (L_g L_f^{r-1} h)(x)  \in \GL_m(\R).&
\end{align*}
If~\eqref{eq:Sys} has relative degree~$r$, then, under the additional assumptions provided in~\cite[Cor.~5.6]{ByrnIsid91a}, system~\eqref{eq:Sys} can be transformed into Byrnes-Isidori form. We assume existence of this transformation in the following, but emphasize that its knowledge is not required for the controller design~-- it is only a tool for the proof of Theorem~\ref{Thm:FunnnelMPC}.

\begin{ass}\label{Ass1} System~\eqref{eq:Sys} has relative degree~$r$ and there exists a diffeomorphism~$\Phi:\R^n\to\R^n$ such that the coordinate transformation $(y(t), \dot y(t),\ldots, y^{(r-1)}(t),\eta(t)) = \Phi(x(t))$ puts the system~\eqref{eq:Sys} into Byrnes-Isidori form
\begin{subequations}\label{eq:BIF}
    \begin{align}
        y^{(r)}(t) &= p\big( y(t), \dot y(t),\ldots, y^{(r-1)}(t),\eta(t)\big)\notag \\
        &\quad + \gamma\big( y(t), \dot y(t),\ldots, y^{(r-1)}(t),\eta(t)\big)\,u(t),
        \label{eq:output_dyn}\\
        \dot \eta(t) &= q\big( y(t), \dot y(t),\ldots, y^{(r-1)}(t),\eta(t)\big),\label{eq:zero_dyn}
    \end{align}
\end{subequations}
where $p:\R^n\to \R^{m}$, $q: \R^n\to \R^{n-rm}$, $\gamma = L_g L_f^{r-1} h:\R^n\to\R^{m\times m}$ are continuously differentiable and $(y(t^0),\dot y(t^0),\ldots, y^{(r-1)}(t^0),\eta(t^0)) =  \Phi(x^0)$. 
\end{ass}

Note that under Assumption~\ref{Ass1} the derivatives of the output~$y$ of~\eqref{eq:Sys} are given by
$y^{(i)}(t) = (L_f^i h)(x(t))$ for $i=0,\ldots,r-1$. In virtue of this we define the map
\begin{equation}\label{eq:chi}
 \hspace*{-2mm}   \chi\!:\!\R^n\!\to\!\R^{rm},\, x\!\mapsto\! \big( h(x), (L_fh)(x),\ldots, (L_f^{r-1}h)(x)\big).
\end{equation}
We further require the following assumption. 

\begin{ass}\label{Ass2}
The internal dynamics~\eqref{eq:zero_dyn} satisfy the following \emph{bounded-input, bounded-state} (BIBS) condition:
\begin{multline}\label{eq:BIBO-ID}
        \fa c_0 >0  \ex c_1 >0  \fa t^0\ge 0 \fa  \eta^0\in\R^{n-rm} \\
       \fa  \zeta\in L^\infty_{\loc}([t^0,\infty),\R^{rm}):\ \Norm{\eta^0}+
        \SNorm{\zeta}  \leq c_0\\ \implies\ \SNorm{\eta (\cdot;t^0,\eta^0,\zeta)} \leq c_1,
\end{multline}
where $\eta (\cdot;t^0, \eta^0,\zeta):[t^0,\infty)\to\R^{n-rm}$ denotes
the unique global solution of~\eqref{eq:zero_dyn} when $(y,\ldots,y^{(r-1)})$ is substituted by~$\zeta$. Note that in view of condition~\eqref{eq:BIBO-ID} the maximal solution $\eta (\cdot;t^0,
\eta^0,\zeta)$ can indeed be extended to a global solution.
\end{ass}

\begin{defn}
We say that the system~\eqref{eq:Sys} belongs to the system class $\cN^{m,r}$, written $(f, g, h) \in\cN^{m,r}$, if it satisfies Assumptions~\ref{Ass1} and~\ref{Ass2}.
\end{defn}

\subsection{Control objective} \label{Sec:CO}

The objective is to design a control strategy such that, with reference to Fig.~\ref{Fig:funnel}, for a given reference
trajectory~$y_{\rf}\in W^{r,\infty}(\Rp,\R^{m})$ the tracking error~$t\mapsto e(t):=y(t)-y_{\rf}(t)$ evolves within the prescribed performance funnel
\begin{align*}
    \cF_\psi:= \setdef{(t,e)\in \Rp\times\R^{m}}{\Norm{e} < \psi(t)}.
\end{align*}
This funnel is determined as $\psi = \psi_1$ by the solution of the following system of differential equations
\begin{equation}\label{eq:funnels}
\begin{aligned}
    \dot \psi_i(t) &= - \alpha_i \psi_i(t) + \beta_i +  p_i \left( \psi_{i+1}(t)- \tfrac{\beta_{i+1}}{\alpha_{i+1}}\right),\\
    \psi_i(0)&= \psi_i^0,\qquad\qquad\qquad i=1,\ldots,r-1,\\
    \dot \psi_r(t) &= -\alpha_r \psi_r(t) + \beta_r,\quad
    \psi_r(0) = \psi_r^0,
\end{aligned}
\end{equation}
where the design parameters
\begin{equation}\label{eq:FC-param}
\begin{aligned}
    &\alpha_1\!>\!\alpha_2\!>\!\ldots\!>\!\alpha_r\!>\!0,\ p_i\!>\!1\ \text{ for } i=1,\ldots,r\!-\!1, \\
    &\beta_i\!>\!0,\  \psi_i^0\!>\!\tfrac{\beta_i}{\alpha_i}\ \text{ for } i=1,\ldots,r
\end{aligned}
\end{equation}
can be chosen as desired. Typically, the specific application dictates the constraints on the tracking error and thus indicates suitable choices for those parameters.
\vspace*{-2mm}
 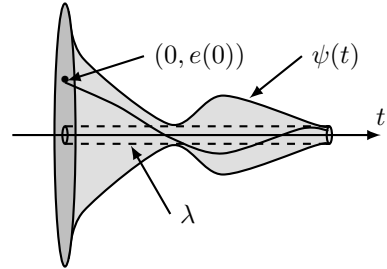
\begin{figure}[h]
  \begin{center}
\begin{tikzpicture}[scale=0.35]
\tikzset{>=latex}
  \filldraw[color=gray!25] plot[smooth] coordinates {(0.15,4.7)(0.7,2.9)(4,0.4)(6,1.5)(9.5,0.4)(10,0.333)(10.01,0.331)(10.041,0.3) (10.041,-0.3)(10.01,-0.331)(10,-0.333)(9.5,-0.4)(6,-1.5)(4,-0.4)(0.7,-2.9)(0.15,-4.7)};
  \draw[thick] plot[smooth] coordinates {(0.15,4.7)(0.7,2.9)(4,0.4)(6,1.5)(9.5,0.4)(10,0.333)(10.01,0.331)(10.041,0.3)};
  \draw[thick] plot[smooth] coordinates {(10.041,-0.3)(10.01,-0.331)(10,-0.333)(9.5,-0.4)(6,-1.5)(4,-0.4)(0.7,-2.9)(0.15,-4.7)};
  \draw[thick,fill=lightgray] (0,0) ellipse (0.4 and 5);
  \draw[thick] (0,0) ellipse (0.1 and 0.333);
  \draw[thick,fill=gray!25] (10.041,0) ellipse (0.1 and 0.333);
  \draw[thick] plot[smooth] coordinates {(0,2)(2,1.1)(4,-0.1)(6,-0.7)(9,0.25)(10,0.15)};
  \draw[thick,->] (-2,0)--(12,0) node[right,above]{\normalsize$t$};
  \draw[thick,dashed](0,0.333)--(10,0.333);
  \draw[thick,dashed](0,-0.333)--(10,-0.333);
  \node [black] at (0,2) {\textbullet};
  \draw[->,thick](4,-3)node[right]{\normalsize$\lambda$}--(2.5,-0.4);
  \draw[->,thick](3,3)node[right]{\normalsize$(0,e(0))$}--(0.07,2.07);
  \draw[->,thick](9,3)node[right]{\normalsize$\psi(t)$}--(7,1.4);
\end{tikzpicture}
\end{center}
 \vspace*{-2mm}
 \caption{Error evolution in a funnel $\mathcal F_{\psi}$ with boundary $\psi(t)$.}
 \label{Fig:funnel}
 \end{figure}
\vspace*{-4mm}

\section{Funnel MPC scheme}

In this section we define the novel FMPC algorithm, which extends~\cite[Alg.~2.7]{BergDenn21} to systems with arbitrary relative degree, and we prove that it is initially and recursively feasible. To this end, we first define, for any solution $(\psi_1,\ldots,\psi_r)$ of~\eqref{eq:funnels}, $y_{\rf}\in W^{r,\infty}(\Rp,\R^{m})$, $t\ge 0$ and $\zeta = (\zeta_1,\ldots,\zeta_r)\in\R^{rm}$,
\begin{equation}\label{eq:aux-ei}
\begin{aligned}
    e_1(t,\zeta) &:= \zeta_1 - y_{\rm ref}(t),\\
    e_{i+1}(t,\zeta) &:= \zeta_{i+1} - y_{\rm ref}^{(i)}(t) + k_i(t,\zeta) e_i(t,\zeta),\\
    k_i(t,\zeta) &:= \left(1- \tfrac{\|e_i(t,\zeta)\|^2}{\psi_i(t)^2}\right)^{-1}
\end{aligned}
\end{equation}
for $i=1,\ldots,r-1$. Then we propose, with design parameter ${\lambda_u\in\Rp}$, the new \textit{stage cost function} $\ell$ defined in~\eqref{eq:stageCostFunnelMPC}. The terms~$\frac {1}{1- \Norm{e_i(t,\zeta)}^2/\psi_i(t)^2}$ penalize the distance of the auxiliary error variables~$e_i$ defined in~\eqref{eq:aux-ei}  to the funnel boundaries~$\psi_i$, whereas the parameter $\lambda_u$ influences the penalization of the control input. Note that $e_1 = y -y_{\rm ref}$.

\begin{figure*}[h!tb]
\begin{equation}\label{eq:stageCostFunnelMPC}
\ell:\Rp\times\R^{rm}\times\R^{m}\to\R\cup\{\infty\}, \
        (t,\zeta,u) \mapsto
        \begin{cases}
           \sum\limits_{i=1}^r  \frac {1}{1- \Norm{e_i(t,\zeta)}^2/\psi_i(t)^2}-r + \lambda_u \Norm{u}^2 & \Norm{e_i(t,\zeta)} \!\neq\! \psi_i(t)\ \forall\, i=1,\ldots,r\\
            \infty,&\text{else}.
        \end{cases}
\end{equation}
\vspace{-8mm}
\end{figure*}

The cost function~$\ell$ is motivated by the following recent result on funnel control
from~\cite[Cor.~3.3]{Berg22}, which is tailored to the present framework.

\begin{prop}\label{Prop:FC}
Consider a system~\eqref{eq:BIF} which satisfies condition~\eqref{eq:BIBO-ID} and $\gamma(x)\in\GL_m(\R)$ for all $x\in\R^n$. Choose $t^0\in\Rp$, funnel design parameters as in~\eqref{eq:FC-param} and let $(\psi_1,\ldots,\psi_r)$ be a global solution of~\eqref{eq:funnels}. Then for all ${K,\xi >0}$ there exist $\hat \eps_1,\ldots,\hat \eps_r\in (0,1)$ such that for all $\eps_i\in [\hat \eps_i, 1)$, $i=1,\ldots,r$ there exists $M>0$ such that
\begin{itemize}
   \item for all $y_{\rm ref}\in W^{r,\infty}([t^0\infty),\R^m)$ with $\|y_{\rm ref}^{(i)}\|_\infty \le K$, $i=0,\ldots,r$,
  \item for all $y^i\in\R^m$ with $\|e_i(t^0,y^1,\ldots,y^r)\|\le \eps_i \psi_i(t^0)$ for $i=1,\ldots,r$, and
  \item for all $\eta^0\in\R^{n-rm}$ with $\|\eta^0\| \le \xi$ and $\hat \eta:= \eta(t^0;0,\eta^0,\zeta)$ for some $\zeta\in C([0,t^0],\R^{rm})$ with $\|e_i(t,\zeta_1(t),\ldots,\zeta_r(t))\|\le \eps_i \psi_i(t)$ for all $t\in [0,t^0]$ and $\zeta_i(t^0) = y^i$ for $i=1,\ldots,r$,
\end{itemize}
the application of the controller
\begin{align*}
    u(t) &= -k_r(t,Y(t)) \gamma\big(Y(t),\eta(t)\big)^{-1} e_r(t,Y(t)),\\
    Y(t) &= (y(t),\ldots,y^{(r-1)}(t)),
\end{align*}
to~\eqref{eq:BIF}, where 
$k_r$ is defined as in~\eqref{eq:aux-ei}, leads to a closed-loop initial value problem with initial conditions $y^{(i-1)}(t^0) = y^i$ for $i=1,\ldots,r$, $\eta(t^0) = \hat \eta$, which has
a solution, every solution can be maximally extended and every maximal solution~$(y,\eta):[t^0,\omega)\to\R^{m}$, $\omega\in(t^0,\infty]$, is global (i.e., $\omega=\infty$) and satisfies
\begin{enumerate}[(i)]
  \item $y\in W^{r,\infty}([t^0,\infty),\R^m)$ and $k_i\in L^\infty([t^0,\infty),\R)$ for $i=1,\ldots,r$;
  \item $u\in L^\infty([t^0,\infty),\R^m)$ with $\|u(t)\|\le M$ for all $t\ge t^0$,
  \item $\|e_i(t,y(t),\ldots,y^{(r-1)}(t))\|\le \eps_i \psi_i(t)$ for all $t\ge t^0$ and all $i=1,\ldots,r$.\hfill \raggedright $\diamond$
\end{enumerate}
\end{prop}

Note that compared to~\cite[Cor.~3.3]{Berg22} the parameter $\xi$, on which $M$ depends, is new and defines a bounded set for the initial values of the internal dynamics. Nevertheless, the proof from~\cite{Berg22} can still be applied when the operators $T_{\eta^0} : \zeta \mapsto \eta(\cdot;0,\eta^0,\zeta)$ are considered and it is observed that by~\eqref{eq:BIBO-ID} a uniform bound for those operators (depending on~$\xi$) on any bounded set in $L^\infty_{\loc}(\Rp,\R^{rm})$ is provided.

Further note that the proof of~\cite[Cor.~3.3]{Berg22} is constructive and explicit expressions for the numbers $\hat \eps_1,\ldots,\hat \eps_r$ and $M = M(\eps_1,\ldots,\eps_r,K,\xi)$ are given, which we do not repeat here (and which require a slight but straightforward modification utilizing $N(s) = -s$ and $\sup_{x\in C} \|\gamma(x)^{-1}\|$ over an appropriate compact set $C\subseteq\R^n$).

Based on the cost function~$\ell$ from~\eqref{eq:stageCostFunnelMPC} and inspired by Proposition~\ref{Prop:FC}, we may define the FMPC algorithm as follows.

\begin{algo}[FMPC]\label{Algo:MPCFunnelCost}\ \\
    \textbf{Given:} System~\eqref{eq:Sys},  funnel design parameters as in~\eqref{eq:FC-param} and a global solution $(\psi_1,\ldots,\psi_r)$ of~\eqref{eq:funnels}, reference signal ${y_{\rf}\in W^{r,\infty}(\Rp,\R^{m})}$, $M>0$, $\eps=(\eps_1,\ldots,\eps_r)\in(0,1)^r$, $t^0\in\Rp$ and
    \begin{align}\label{eq:DefD0}
        x^0 \in \cD^{\ve}_{t^0}:= \setdef{x\in\R^n\!}{\!\!\!\begin{array}{l} \|e_i(t^0,\chi(x))\|\le \eps_i \psi_i(t^0)\\
        \text{for all } i=1,\ldots,r \end{array}\!\!\!}
    \end{align}
     for $\chi$ as in~\eqref{eq:chi}, and stage cost function~$\ell$ as in~\eqref{eq:stageCostFunnelMPC}.\\
    \textbf{Set} the time shift $\delta >0$, the prediction horizon $T\geq\delta$ and initialize the current time
        $\hat t :=t^0$.\\
    \textbf{Steps:}
    \begin{enumerate}[(a)]
    \item\label{agostep:FunnelMPCFirst}
        Obtain a measurement of the state at~$\hat t$ and set $\widehat x :=x(\hat t)$.
    \item
        Compute a solution $u^{\star}\in L^\infty([\hat t,\hat t +T],\R^{m})$ of the Optimal
        Control Problem (OCP)
    \begin{equation}\label{eq:OCP}
    \begin{alignedat}{2}
            &\!\mathop
            {\operatorname{minimize}}_{u\in L^{\infty}([\hat t,\hat t+T],\R^{m})}  && \
                \int_{\hat t}^{\hat t+T} \ell\big(t,\zeta(t),u(t)\big) {\rm d}t \\
            &\quad \text{subject to}       &\ \ & \zeta(t) = \chi\big(x(t;\hat t, \widehat x,u)\big),     \\
            &                        && \Norm{u(t)}  \leq M\ \text{ for } t\in [\hat t,\hat t +T],\\
            &                        && \|e_i(\hat t+\delta,\zeta(\hat t+\delta))\| \le  \eps_i \psi_i(\hat t+\delta),\\
            &&& i=1,\ldots,r
    \end{alignedat}
    \end{equation}
    \item Apply the feedback law
        \begin{equation}\label{eq:FMPC-fb}
            \mu:[\hat t,\hat t+\delta)\times\R^n\to\R^m, \quad \mu(t,\widehat x) =u^{\star}(t)
        \end{equation}
        to system~\eqref{eq:Sys}.
        Increase $\hat t$ by $\delta$ and go to Step~\eqref{agostep:FunnelMPCFirst}.

    \end{enumerate}
\end{algo}

Note that in the OCP~\eqref{eq:OCP} the last~$r$ inequalities constitute a feasibility constraint on the output~$y$ and its first~${r-1}$ derivatives, which resembles the constraint used in~\cite[Eq.~(9)]{berger2019learningbased}.

In the following main result we show that for suitable $M>0$ and $\eps \in (0,1)^r$ the FMPC Algorithm~\ref{Algo:MPCFunnelCost} is initially and recursively feasible for every prediction horizon ${T>0}$ and that it guarantees the evolution of the tracking error within the performance funnel $\cF_{\psi_1}$.

\begin{thm}\label{Thm:FunnnelMPC}
Consider a system~\eqref{eq:Sys} with $(f,g,h)\in\cN^{m,r}$. Choose funnel design parameters as in~\eqref{eq:FC-param} and let
$(\psi_1,\ldots,\psi_r)$ be a global solution of~\eqref{eq:funnels}. Let 
\begin{itemize}
\item $K,\xi >0$, $\eps = (\eps_1,\ldots,\eps_r)\in(0,1)^r$ and $M = M(\eps,K,\xi)$ as in Prop.~\ref{Prop:FC},
\item $y_{\rf}\in W^{r,\infty}(\Rp,\R^{m})$  such that $\|y_{\rm ref}^{(i)}\|_\infty \le K$
for $i=0,\ldots,r$,
\item $t^0\in\Rp$ and $B\subset\cD^{\ve}_{t^0}$ be a bounded set such that for all $x^0\in B$ we have that $(\zeta^0,\eta^0) = \Phi(x^0)$ satisfies $\|\eta^0\|
\le \xi$.
\end{itemize}
Then the FMPC
    Algorithm~\ref{Algo:MPCFunnelCost} with $\delta>0$ and $T\ge \delta$ is initially and
    recursively feasible for every $x^0\in B$, i.e., at time $\hat t = t^0$ and at each
    successor time $\hat t\in t^0+\delta\N$ the OCP~\eqref{eq:OCP}
    has a solution. In particular, the closed-loop system consisting of~\eqref{eq:Sys} and the FMPC feedback~\eqref{eq:FMPC-fb} has a (not necessarily unique) global solution $x:[t^0,\infty)\to\R^n$ and the corresponding input is given by
    \[
        u_{\rm FMPC}(t) = \mu(t,x(\hat t)),\quad t\in [\hat t,\hat t+\delta),\ \hat t\in t^0+\delta\N.
    \]
    Furthermore, each global solution~$x$ with corresponding input $u_{\rm FMPC}$ satisfies:
    \begin{enumerate}[(i)]
        \item\label{th:item:BoundedInput}
$\fa t\ge t^0:\ \Norm{u_{\rm FMPC}(t)}\leq M$.
        \item\label{th:item:ErrorInFunnel} $\fa t\ge t^0:\ \|e_i(t,\chi(x(t)))\| < \psi_i(t)$; in particular the error $e=y-y_{\rf}$ evolves within the funnel
            $\cF_{\psi_1}$, i.e., $\Norm{e(t)} < \psi_1(t)$ for all $t\ge t^0$.
    \end{enumerate}
\end{thm}

The proof is relegated to Appendix~\ref{App:PROOF}.

\section{Simulation}\label{Sec:Sim}

To illustrate the proposed FMPC scheme, we consider the mass-on-car system introduced in~\cite{SeifBlaj13}, where on a car with mass~$m_1$ (in kg) a ramp is mounted on which a mass~$m_2$ (in kg), coupled to the car by a spring-damper-component with spring constant~$k > 0$ (in N/m) and damping~$d>0$ (in Ns/m), passively moves; a control force~$F = u$ (in N) can be applied to the car.
The situation is depicted in Fig.~\ref{Fig:Mass-on-a-car}.
\begin{figure}[h!]
\begin{center}
\includegraphics[trim=2cm 4cm 5cm 15cm,clip=true,width=5.5cm]{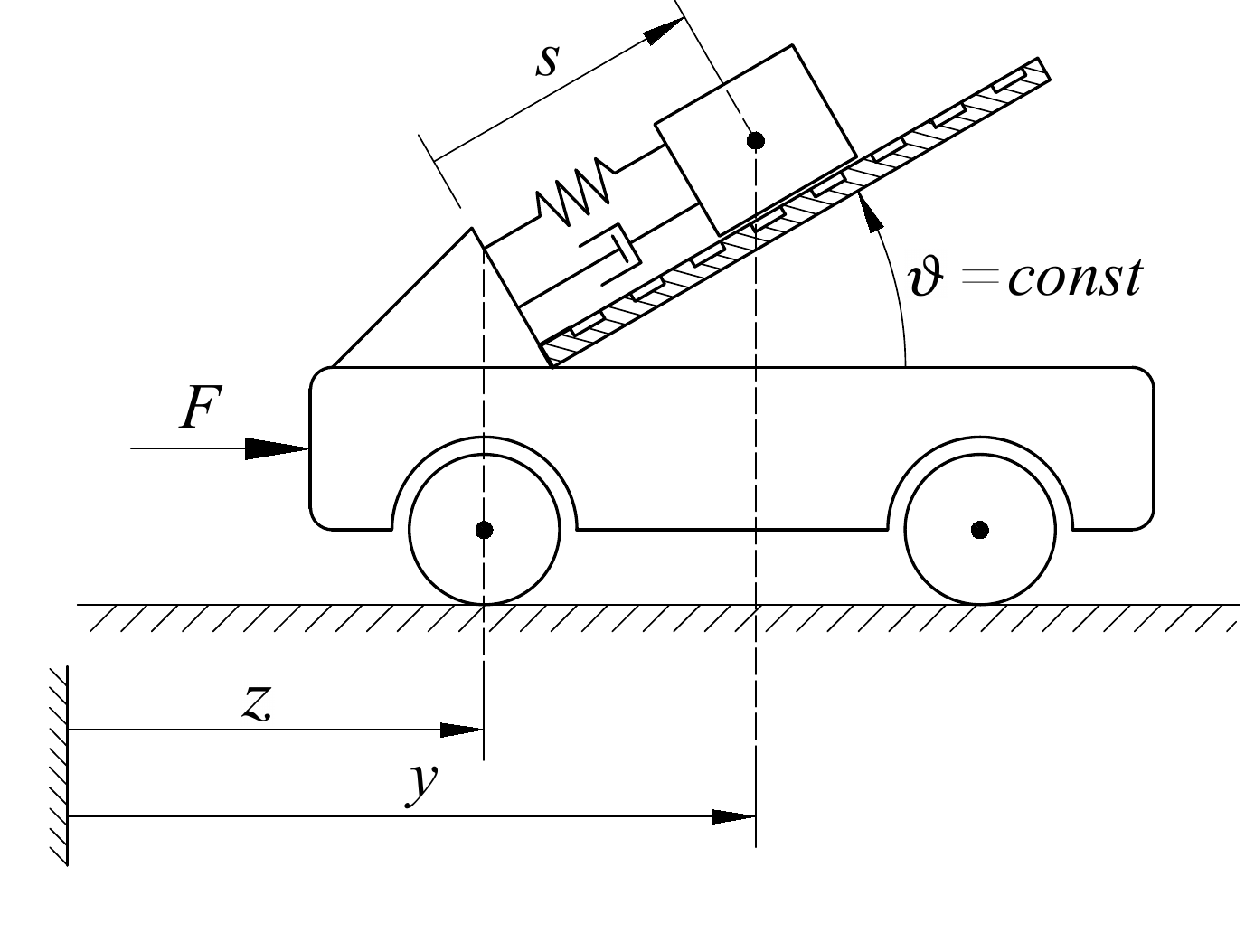}
\end{center}
\vspace{-4mm}
\caption{Mass-on-car system.}
\label{Fig:Mass-on-a-car}
\end{figure}
The equations of motion for the system read
\begin{subequations} \label{eq:Mass-on-car-system}
\begin{equation}  \label{eq:Mass-on-car}
\begin{bmatrix}
m_1 + m_2 \!&\! m_2 \cos(\vartheta)
\\ m_2 \cos(\vartheta) \!&\! m_2
\end{bmatrix}\!\!
\begin{pmatrix}
\ddot z(t) \\
\ddot s(t)
\end{pmatrix} \!+\!
\begin{pmatrix}
0 \\
ks(t) \!+\! d \dot s(t)
\end{pmatrix} \!=\!
\begin{pmatrix}
u(t) \\
0
\end{pmatrix}\!\!,
\end{equation}
with the horizontal position of the second mass~$m_2$ as output
\begin{equation} \label{eq:Mass-on-car-output}
y(t) = z(t) + \cos(\vartheta) s(t).
\end{equation}
\end{subequations}
For the simulation we choose the parameters $m_1 = 4$, ${m_2 = 1}$, $k=2$, $d = 1$, $\vartheta = \pi/4$ and the initial values $z(0)=s(0)=\dot z(0) = \dot s(0) = 0$.
The objective is tracking of the reference signal~$y_{\rm ref}: \rp \to \R$, $t \mapsto \cos(t)$ so that the error $e(t) = y(t) - y_{\rm ref}(t)$ satisfies $\|e(t)\|\le \psi_1(t)$ for the solution $(\psi_1,\psi_2)$ of~\eqref{eq:funnels} for the parameters
\begin{align*}
    \alpha_1 &= 1.5, \alpha_2 = 0.9\cdot \alpha_1, \beta_1 = 0.15, \beta_2 = 0.5\cdot \alpha_2, \\
    p_1 &=  1.1, \psi_1^0 = 4.1, \psi_2^0 = 2,
\end{align*}
which are chosen as in~\cite{Berg22}. As outlined in~\cite[Sec.~3]{BergIlch21} for the above parameters system~\eqref{eq:Mass-on-car-system}
belongs to the class $\cN^{1,2}$, in particular the relative degree is two.
We compare the FMPC Algorithm~\ref{Algo:MPCFunnelCost} with OCP~\eqref{eq:OCP}
to the FMPC scheme from~\cite{BergDenn21}.
For Algorithm~\ref{Algo:MPCFunnelCost} we choose, according to the procedure provided in the proof of~\cite[Thm.~3.2]{Berg22} and rounded to the second decimal
place, $\eps_1 = 0.94$ and $\eps_2 = 0.99$. Since the simulation of FMPC in~\cite{BergDenn21} generated control values below~15, we choose $M=15$. Due to discretisation, only step functions with constant step length $0.04$ are considered for the
OCP~\eqref{eq:OCP}. The prediction horizon
and  time shift are selected as $T=0.6$ and $\delta=0.04$, resp. We further choose the parameter
$\lambda_u=\tfrac{1}{100}$ for the stage cost~$\ell$. The parameters $T$, $\delta$ and $\lambda_u$ are chosen as in \cite{BergDenn21}. All simulations are performed
on the time interval $[0,10]$ with the MATLAB routines \texttt{ode45}  and \texttt{fmincon}
and are depicted in Fig.~\ref{Fig:SimulationFunnelMPC}.
Fig.~\ref{Fig:SimulationOutputError} shows the tracking error due to the two different FMPC schemes
evolving within the funnel boundaries given by~$\psi_1$, while
the respective control signals are displayed in
Fig.~\ref{Fig:SimulationControlInput}.
It is evident that both control schemes achieve the evolution of the tracking error within the
performance boundaries given by~$\psi_1$. However, the FMPC Algorithm~\ref{Algo:MPCFunnelCost} with
OCP~\eqref{eq:OCP} requires less input action than the FMPC scheme
from~\cite{BergDenn21}. This superior performance is a consequence of the fact that the stage cost~$\ell$ not only penalizes the distance of the error~$e=e_1$ to~$\psi_1$, but also the distance of~$e_2 = \dot e + k_1 e$ to~$\psi_2$.
\begin{figure}[ht] \centering
    \begin{subfigure}[b]{0.45\textwidth}
     \centering
        \includegraphics[width=8.0cm]{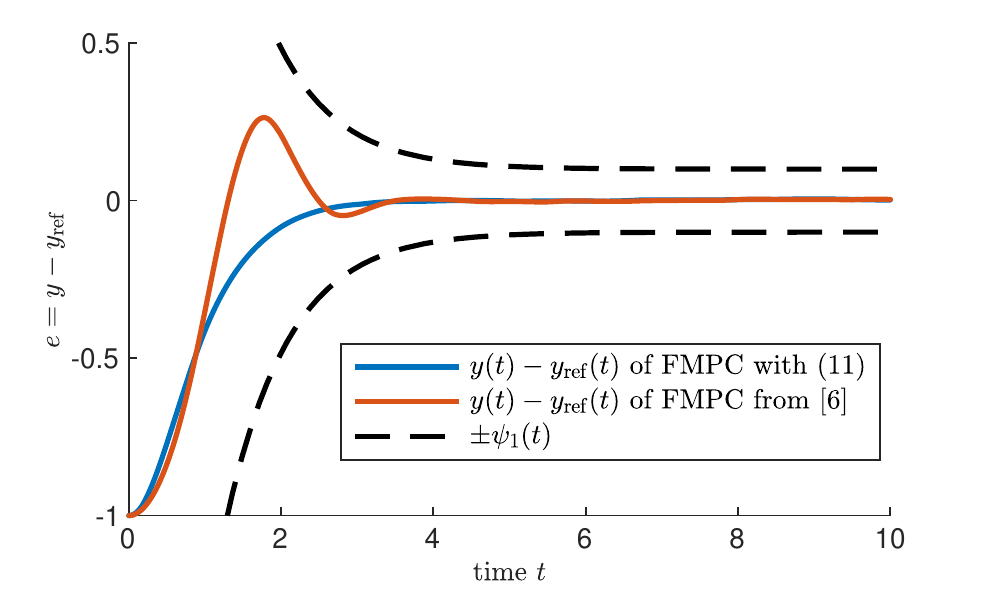}
        \caption{Tracking error~$e$ and funnel boundary~$\psi_1$}
        \label{Fig:SimulationOutputError}
    \end{subfigure}
\end{figure}
\begin{figure}[ht]\ContinuedFloat
    \centering
    \begin{subfigure}[b]{0.45\textwidth}
        \centering
        \includegraphics[width=8.0cm]{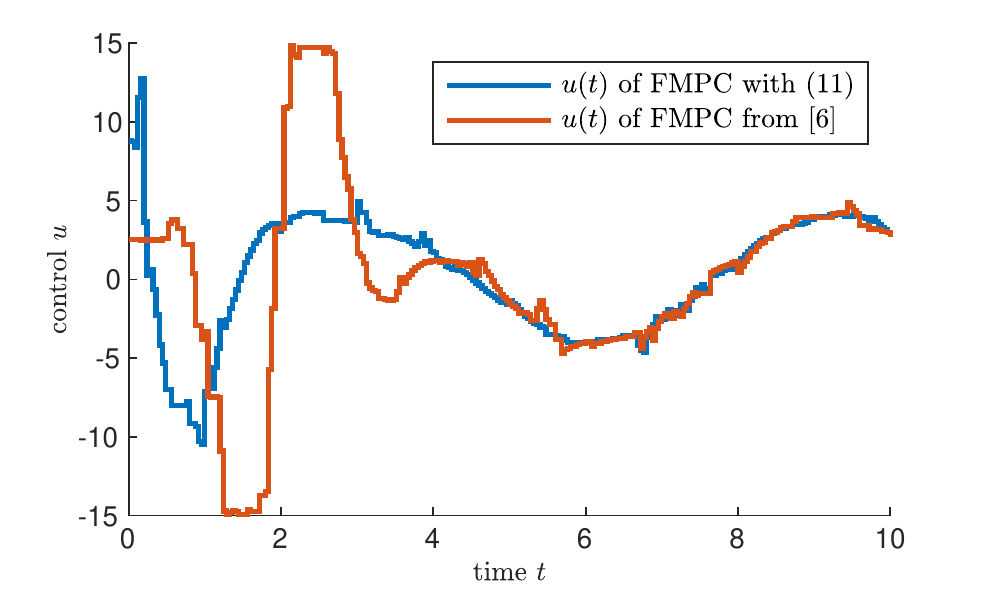}
        \caption{Control input}
        \label{Fig:SimulationControlInput}
    \end{subfigure}
    \caption{Simulation of system~\eqref{eq:Mass-on-car-system} under FMPC Algorithm~\ref{Algo:MPCFunnelCost} and FMPC from~\cite{BergDenn21}}
    \label{Fig:SimulationFunnelMPC}
\end{figure}

\section{Conclusion}

In the present paper we proposed a novel stage cost for FMPC and proved that the resulting FMPC Algorithm~\ref{Algo:MPCFunnelCost} is initially and recursively feasible. This extends earlier approaches from~\cite{berger2019learningbased,BergDenn21} to nonlinear systems with arbitrary relative degree and belonging to the system class~$\cN^{m,r}$. Although we didn't require any terminal conditions or a sufficiently long prediction horizon (as in~\cite{Denn22} for relative degree two), we imposed an additional feasibility constraint in the OCP~\eqref{eq:OCP}. This constraint does not only restrict the set of admissible controls, but the required parameters $\eps_1,\ldots,\eps_r\in(0,1)$ and $M>0$ provided by Proposition~\ref{Prop:FC} are usually quite conservative and hard to compute. Further research should focus on relaxing the OCP by removing the feasibility constraints.

\appendix

\subsection{A preliminary lemma}

\numberwithin{thm}{subsection}

\begin{lem}\label{Lem:appendix}
Consider a system~\eqref{eq:Sys} with ${(f,g,h)\in\cN^{m,r}}$. Choose funnel design parameters as in~\eqref{eq:FC-param} and let
$(\psi_1,\ldots,\psi_r)$ be a global solution of~\eqref{eq:funnels}. Let $t^0\in\Rp$, $T>0$ and $y_{\rm ref}\in W^{r,\infty}([t^0,\infty),\R^m)$. Further let ${x^0\in\R^n}$ be such that $\|e_i(t^0,\chi(x^0))\|<\psi_i(t^0)$ for all $i=1,\ldots,r$. Then there exists $Y>0$ such that for all ${u\in L^\infty([t^0,t^0+T],\R^m)}$ for which $x(t;t^0,x^0,u)$ satisfies~\eqref{eq:Sys} and $\|e_i(t,\chi(x(t;t^0,x^0,u)))\|<\psi_i(t)$ for all ${t\in[t^0,t^0+T]}$ and $i=1,\ldots,r$, we have $\|y^{(i-1)}(t)\|\le Y$ for all $t\in[t^0,t^0+T]$ and $i=1,\ldots,r$.
\end{lem}
\begin{proof}
For brevity we identify $e_i(t) = e_i(t,\chi(x(t;t^0,x^0,u)))$ and observe that the relations~\eqref{eq:aux-ei} imply that the differential equations
\[
    \dot e_i(t) = e_{i+1}(t) - \gamma_i(t) - \dot \gamma_{i-1}(t),\quad i=1,\ldots,r-1,
\]
where $\gamma_0(t) = 0$ and $\gamma_i(t) = \frac{e_i(t)}{1-\|e_i(t)\|^2/\psi_i(t)^2}$, are satisfied. From those it follows from a repetition of Steps~3--6 and~9 of the proof of~\cite[Thm.~3.1]{Berg22} that there exist ${\eps_1,\ldots,\eps_{r-1}\in (0,1)}$, which only depend on~$x^0$ and the parameters in~\eqref{eq:FC-param}, such that $\|e_i(t)\|\le \eps_i \psi_i(t)$ for all $t\in[t^0,t^0+T]$ and $i=1,\ldots,r-1$. From this,~\eqref{eq:aux-ei}, the assumption and the monotony of~$\psi_i$ it follows directly that
\[
    \|y^{(i-1)}(t)\| < \psi_{i}(t^0) + \tfrac{\psi_{i-1}(t^0)}{1-\eps_{i-1}^2} + \|y_{\rm ref}^{(i-1)}\|_\infty =: Y_i
\]
for all $t\in[t^0,t^0+T]$ and $i=1,\ldots,r$, where $\psi_0 := 0$ and $\eps_0 := 0$. With $Y:= \max_{i=1,\ldots,r} Y_i$ the proof is complete.
\end{proof}

\subsection{Proof of Theorem~\ref{Thm:FunnnelMPC}}\label{App:PROOF}
Let $T\ge \delta$ be arbitrary but fixed.
For $t\geq t^0$ we define in addition to $\cD^{\ve}_{t}$ as in~\eqref{eq:DefD0} the set
\[
     \cD_{t}:= \setdef{x\in\R^n}{ \|e_i(t,\chi(x))\| < \psi_i(t)\
        \fa i=1,\ldots,r}.
\]
For $\hat{t}\geq 0$ we denote by $I_{\hat{t}}^T$ the interval $[\hat{t},\hat{t}+T]$ and further for
$\hat{x}\in\cD_{\hat{t}}$ by $\cU(\hat{t},\hat{x})$ the set
\begin{align*}\label{eq:SetControls}
    \setdef
    {\!\!u\!\in\! L^\infty(I_{\hat{t}}^T,\R^m)\!\!}
    { \!\!
        \begin{array}{l}
                x(t;\hat{t},\hat{x},u)\text{ satisfies \eqref{eq:Sys} and}\\
                x(t;\hat{t},\hat{x},u)\in\cD_t\text{ for all } t\in I_{\hat{t}}^T,\\
                x(\hat{t}+\delta;\hat{t},\hat{x},u)\in\cD^{\ve}_{\hat{t}+\delta}, \SNorm{u}\!\!\le M
        \end{array}
    \!\!\!\!}\!.
\end{align*}
This is the set of all $L^\infty$-controls $u$ bounded by $M$ which, if applied to
system~\eqref{eq:Sys}, guarantee that the error signals~$e_i(t,\chi(x(t;\hat{t},\hat{x},u)))$ evolve
within their respective funnels on the interval $I_{\hat{t}}^T$ and moreover
${x(\hat{t}+\delta;\hat{t},\hat{x},u)\in\cD^{\ve}_{\hat{t}+\delta}}$.
By Proposition~\ref{Prop:FC} we have that $\cU(t^0,x^0)\neq\emptyset$ for all $x^0\in B$. Furthermore, for any $\hat t\ge t^0$ we have that $\cU(\hat{t},x(\hat{t};t^0,x^0,u))\neq\emptyset$ for all $u\in L^\infty([t^0,\hat t],\R^m)$ such that $x(\hat{t};t^0,x^0,u)\in \cD^{\ve}_{\hat{t}}$.

In the following we show that if $\cU(\hat{t},\hat{x})$ is non-empty for some $\hat t\ge t^0$ and $\hat x\in\cD^{\ve}_{\hat{t}}$, then the
OCP~\eqref{eq:OCP} has a solution~$u^{\star}\in\cU(\hat{t},\hat{x})$~---
this proves the theorem. To this end, we assume $\hat{x}\in\cD^{\ve}_{\hat{t}}$ in the following. The proof consists of several steps and follows the idea of~\cite[Thms.~4.3~\&~4.6]{BergDenn21}.

\textit{Step 1:}
We show that for $u\in\cU(\hat{t},\hat{x})$, the function $\ell\big(\cdot,\zeta(\cdot),u(\cdot)\big)$ with
$\zeta(\cdot)=\chi(x(\cdot;\hat{t},\hat{x},u))$ is positive on $I_{\hat{t}}^T$ and
$\int_{I_{\hat{t}}^T}\ell(t,\zeta(t),u(t)){\rm d}{t}<\infty$.
By $u\in\cU(\hat{t},\hat{x})$ we have $x(t;\hat{t},\hat{x},u)\in\cD_t$ for all $t\in I_{\hat{t}}^T$.
Therefore, $\Norm{e_i(t,\zeta(t))}^2<\psi_i(t)^2$ for all $t\in I_{\hat{t}}^T$ and all
$i=1,\ldots,r$. Due to the compactness of $I_{\hat{t}}^T$ and the continuity of $\zeta$, $e_i$,
$\psi_i$, there exists $\delta>0$
with $\Norm{e_i(t,\zeta(t))}^2/\psi_i(t)^2<1-\delta$ for all $t\in I_{\hat{t}}^T$ and all
$i=1,\ldots,r$.
Hence, $\ell\big(t,\zeta(t),u(t)\big)\geq0$ for all $t\in I_{\hat{t}}^T$ and
\begin{align*}
    \int_{I_{\hat{t}}^T}&\ell(t,\zeta(t),u(t)){\rm d}{t}\!\\
    &=\!\int_{I_{\hat{t}}^T}\!
        \sum\limits_{i=1}^r\tfrac {1}{1- \Norm{e_i(t,\zeta(t))}^2/\psi_i(t)^2}
            -r \!+\! \lambda_u\!\Norm{u(t)}^2{\rm d}{t}\\
    &\leq \int_{I_{\hat{t}}^T}{ \tfrac{r}{\delta}  + \lambda_u \SNorm{u}^2}{\rm d}{t}
    \le \left(\tfrac{r}{\delta} + \lambda_u M^2\right) T
    <\infty.
\end{align*}

\textit{Step 2:}
We show that the set
\begin{align*}
    \setdef
    {\!\!u\!\in\! L^\infty(I_{\hat{t}}^T\!,\R^m)\!\!}
    { \!\!\!
        \begin{array}{l}
            x(t;\hat{t},\hat{x},u)\text{ satisfies \eqref{eq:Sys} for all } t\!\in\! I_{\hat{t}}^T\!,\\
                x(\hat{t}+\delta;\hat{t},\hat{x},u)\in\cD^{\ve}_{\hat{t}+\delta}, \SNorm{u}\!\!\le M,\\
                \int_{I^{T}_{\hat{t}}}\ell(t,\chi(x(t; \hat{t}, \hat{x} ,u)),u(t))\, {\rm d} t < \infty
        \end{array}
    \!\!\!\!}\!,
\end{align*}
denoted by $\tilde{\cU}(\hat{t},\hat{x})$, is a subset of $\cU(\hat{t},\hat{x})$. Let $u\in \tilde{\cU}(\hat{t},\hat{x})$ and set
$\zeta(\cdot)=\chi(x(\cdot;\hat{t},\hat{x},u))$.
The claim is proved by showing $\Norm{e_i(t,\zeta(t))}<\psi_i(t)$ for all $t\in I_{\hat{t}}^T$ and ${i=1,\ldots,r}$.
Since $\hat{x}\in\cD^{\ve}_{\hat{t}}$, we know $\Norm{e_i(\hat{t},\zeta(\hat{t}))}<\psi_i(\hat{t})$.
Assume there exists $t\in I^{T}_{\hat{t}}$ with
$\Norm{e_i(t,\zeta(t))}\geq\psi_i(t)$ for $i=1,\ldots,r$. By continuity of
$e_i$, $\zeta_i$, and $\psi_i$, there exists
\[
    \tilde{t}:=\min\setdef{ t\in I^T_{\hat{t}}}{ \ex i =1,\ldots,r: \Norm{e_i(t,\zeta(t))}=\psi_i(t)}.
\]
Let $j\in\cbl1,\ldots,r\cbr$ with $\Norm{e_j(\tilde{t},\zeta(\tilde{t}))}=\psi_j(\tilde{t})$.
Recalling the definition of the Lebesgue integral, see e.g.~\cite[Def.~11.22]{Rudi76},
$\int_{I^{T}_{\hat{t}}}\ell(t,\zeta(t),u(t))\,
{\rm d} t < \infty$ implies $\int_{I^{T}_{\hat{t}}}(\ell(t,\zeta(t),u(t)))^+\,
{\rm d} t < \infty$ where $(\ell(t,\zeta(t),u(t)))^+:=\max\cbl(\ell(t,\zeta(t),u(t))),0\cbr$.
Note that $\Norm{e_i(\tilde{t},\zeta(\tilde{t}))}<\psi_i(\tilde{t})$ for all $t\in[\hat{t},\tilde{t})$
and for all $i=1,\ldots, r$. Therefore,
\begin{align*}
    &\int_{\hat{t}}^{\tilde{t}}\tfrac {1}{1- \Norm{e_j(t,\zeta(t))}^2/\psi_j(t)^2}{\rm d}{t}
        \leq\int_{\hat{t}}^{\tilde{t}}\sum\limits_{i=1}^r\tfrac {1}{1- \Norm{e_i(t,\zeta(t))}^2/\psi_i(t)^2} {\rm d}{t}\!\\
        &\leq\!\!\int_{I^{T}_{\hat{t}}} \!
        \!\rbl\sum\limits_{i=1}^r\tfrac {1}{1- \Norm{e_j(t,\zeta(t))}^2/\psi_i(t)^2}\rbr^{\!+}\!\! {\rm d}{t}\!\\
    &\leq\!\!\int_{I_{\hat{t}}^T} \!
        \!\rbl \sum\limits_{i=1}^r\tfrac{1}{1- \Norm{e_i(t,\zeta(t))}^2/\psi_i(t)^2}
            -r + \lambda_u\Norm{u(t)}^2\!\!\rbr^{\!+}\!\!{\rm d}{t}+Tr\\
    &=\!\!\int_{I_{\hat{t}}^T}( \ell(t,\zeta(t),u(t)))^+\!{\rm d}{t}+Tr<\infty.
\end{align*}
As continuous functions $\zeta$ and $y^{(i)}_{\rm ref}$
are bounded on the compact interval~$[\hat{t},\tilde{t}]$ for all $i=0,\ldots, r$. For the diffeomorphism~$\Phi$ from Assumption~\ref{Ass1} we have that $\Phi(x(\cdot;\hat t,\hat x,u)) = (\zeta(\cdot),\eta(\cdot))$ on~$I_{\hat{t}}^T$ for some absolutely continuous $\eta:I_{\hat{t}}^T\to\R^{n-rm}$. As a consequence of Assumption~\ref{Ass2}, $\eta$ is bounded
on the interval~$[\hat{t},\tilde{t}]$.
Since the functions~$p$ and~$\gamma$ in \eqref{eq:output_dyn} are continuously differentiable,
$y^{(r)}$ is bounded on~$[\hat{t},\tilde{t}]$ as well, and hence $e^{(i)} = y^{(i)} - y^{(i)}_{\rm ref}$ is bounded for all $i=0,\ldots, r$. By definition of $\tilde t$ we have for $e_i(\cdot) := e_i(\cdot,\zeta(\cdot))$ that
$\Norm{e_i({t})}< \psi_i({t})$ for all $t\in[\hat{t},\tilde{t})$ and all $i=1,\ldots, r$. Then, by the same arguments as in the proof of Lemma~\ref{Lem:appendix}, there exist $\delta_1,\ldots,\delta_{r-1}\in (0,1)$, which only depend on~$\hat x$ and the parameters in~\eqref{eq:FC-param}, such that $\Norm{e_i({t}))}\le \delta_i \psi_i(t)$ for all $t\in[\hat{t},\tilde{t})$ and all $i=1,\ldots,r-1$ (and by continuity the inequality also holds for $t=\tilde t$). Then $k_i(\cdot) := k_i(\cdot,\zeta(\cdot))$ from~\eqref{eq:aux-ei} is bounded on~$[\hat{t},\tilde{t}]$ for $i=1,\ldots,r-1$.
Therefore, since
\begin{align*}
    \ddt \big(k_i(t)e_i(t)\big)\!=&2k_i(t)^2
    \!\rbl\!
    \tfrac{\Norm{e_i(t)}^2}{\psi_i(t)^3}\dot{\psi}_i(t)
    \!+\!\tfrac{e_i(t)^\top\dot{e}_i(t)}{\psi_i(t)^2}
    \!\rbr \!e_i(t)\\&+k_i(t)\dot{e}_i(t)
\end{align*}
for $i=1,\ldots,r-1$ and invoking boundedness of~$\psi_i$ and~$\dot \psi_i$ due to~\eqref{eq:funnels}, it follows by induction and from the relations~\eqref{eq:aux-ei} that $\dot{e}_i(\cdot)$ is essentially
bounded for all $i=1,\ldots,r-1$, where for $i=1$ we have that~$\dot{e}_1  = e_2 - k_1 e_1$ is bounded.  Furthermore, it is straightforward to see that $\dot e_r = e^{(r)} + \ddt (k_{r-1} e_{r-1})$ is bounded. In particular, we have shown that $\dot{e}_j$ is bounded and hence $e_j$ is Lipschitz continuous. Since $\dot \psi_j$ is bounded and $\psi_j(t)\ge \beta_j/\alpha_j$ it is also clear that $\ddt (1/\psi_j) = -\dot \psi_j/\psi_j^2$ is bounded, hence $1/\psi_j$ is Lipschitz continuous. Therefore, $1-\|e_j(\cdot)^2\|/\psi_j(\cdot)^2$ is a Lipschitz continuous function on the
interval~$[\hat{t},\tilde{t}]$, hence it follows from~\cite[Lem.~4.1]{BergDenn21} that  $1-\|e_j(\cdot)^2\|/\psi_j(\cdot)^2$ is strictly positive on the interval
$[\hat{t},\tilde{t}]$, contradicting the definition of~$\tilde t$. Hence
$\tilde{\cU}(\hat{t},\hat{x})\subseteq\cU(\hat{t},\hat{x})$.

\textit{Step 3:}
We show that the OCP~\eqref{eq:OCP} has a solution
${u^{\star}\in\cU(\hat{t},\hat{x})}$. It follows from Step~1 that
$\cU(\hat{t},\hat{x})\subseteq\tilde{\cU}(\hat{t},\hat{x})$ and together with Step~2 we have $\cU(\hat{t},\hat{x})=\tilde{\cU}(\hat{t},\hat{x})\neq \emptyset$, the latter by assumption. Solving the OCP~\eqref{eq:OCP} is
therefore equivalent to minimizing the function
\begin{align*}
    J&:L^\infty(I^{T}_{\hat{t}},\R^{m})\to\R\cup\cbl\infty\cbr,\quad\\
    &\ u\mapsto
        \begin{cases}
            \int_{I^{T}_{\hat{t}}}\ell(t,\chi(x(t;\hat{t},\hat{x},u)),u(t))\, {\rm
            d}t,&u\in\cU(\hat{t},\hat{x})\\
            \infty,&\text{else.}
        \end{cases}
\end{align*}
As a consequence of Step~1, $J(u)\geq0$ for all $u\in\cU(\hat{t},\hat{x})$.
Hence, the infimum $J^{\star}: = \inf_{u \in \cU(\hat{t},\hat{x})} J(u)$ exists.
Let $(u_k)\in(\cU(\hat{t},\hat{x}))^\N$ be a minimizing sequence, meaning $J(u_k)\to J^*$.
Since $L^\infty(I^{T}_{\hat{t}},\R^{m})\subset L^2(I^{T}_{\hat{t}},\R^{m})$ and $\SNorm{u_k}\leq M$
for all $k\in\N$, we conclude that $(u_k)$ is a bounded sequence in the Hilbert space $L^2$.
Hence, there exists $u^{\star}\in L^2(I^{T}_{\hat{t}},\R^{m})$ and a weakly convergent subsequence
${u_k\rightharpoonup u^{\star}}$ (which we do not relabel).
Let $(x_k):=(x(\cdot;\hat{t},\hat{x},u_k))\in C(I^T_{\hat{t}},\R^n)^\N$ be the sequence of
associated responses.
According to Lemma~\ref{Lem:appendix}, there exists $Y>0$ such that $\SNorm{\chi(x_k)}\leq Y$ for all
$k\in\N$. As in Step~2, let $\eta_k:I^{T}_{\hat{t}}\to\R^{n-rm}$ be such that $(\chi(x_k(\cdot)),\eta_k(\cdot))=\Phi(x_k(\cdot))$ and observe that $\eta_k(\cdot) = \eta(\cdot;\hat t, \eta_k(\hat t), \chi(x_k))$.
Since $\|\eta_k(\hat t)\| \le \|\Phi(\hat x)\|$, independent of~$k$, it follows from Assumption~\ref{Ass2} with $c_0 := Y + \|\Phi(\hat x)\|$ that there exists $c_1>0$ such that
$\SNorm{\eta_k}\leq c_1$ for all $k\in\N$.
Therefore, $x_k(t)$ is an element of the compact set
\[
    \Phi^{-1}\!\rbl
    \setdef
    {\!\!
        \begin{pmatrix}
            z_1\\
            z_2
        \end{pmatrix}\!\!
        \in\!\R^{rm}\!\times\!\R^{n-rm}\!
    }
    {\!
        \Norm{z_1}\!\le Y\wedge
        \Norm{z_2}\!\le c_1\!\!
    }
    \rbr
\]
for all $t\in I_{\hat{t}}^T$ and all $k\in\N$. Hence, $(x_k)$ is uniformly bounded.
Then, by a repetition of Steps~2--4 of the proof of~\cite[Thm.~4.6]{BergDenn21}, we may infer
that $(x_k)$ has a subsequence (which we do not relabel) that converges uniformly to $x^{\star}=x(\cdot;\hat{t},\hat{x},u^{\star})$ and $\SNorm{u^\star}\le M$.
Due to the continuity of~$\chi$ and $e_i$, the uniform convergence of~$(x_k)$ implies the pointwise
convergence of $\chi(x_k(\cdot))$ and $e_i(\cdot,\chi(x_k(\cdot)))$ for all $i=1,\ldots,r$. Thus,
$x(\hat{t}+\delta;\hat{t},\hat{x},u^{\star})\in\cD^{\ve}_{\hat{t}+\delta}$. It remains to show that $u^{\star}\in\cU(\hat{t},\hat{x})$ and $J(u^{\star})=J^\star$.
Again this follows along the lines of Steps~5--6 of the proof of~\cite[Thm.~4.6]{BergDenn21} and this completes the proof. \hfill $\blacksquare$

\bibliographystyle{IEEEtran}
\bibliography{references}
\end{document}